\documentclass[11pt]{article}
\usepackage{amsmath,amssymb,amsthm,mathtools,hyperref}
\usepackage[margin=1.1in]{geometry}
\usepackage{xcolor}

\newtheorem{theorem}{Theorem}[section]
\newtheorem{lemma}[theorem]{Lemma}
\newtheorem{proposition}[theorem]{Proposition}

\theoremstyle{definition}
\newtheorem{definition}[theorem]{Definition}
\theoremstyle{remark}

\DeclareMathOperator{\parr}{par}
\DeclareMathOperator{\sib}{sib}

\newcommand{\rise}[2]{#1^{\overline{#2}}}
\newcommand{\TC}{\text{TC}}

\title{A Short Combinatorial Proof of the Pons--Batle Identity for Counting Tree-Child Networks}
\author{Hao Yu and Louxin Zhang\thanks{Department of Mathematics, National University of Singapore, Singapore 119076, Singapore. This work was partially supported by a Singapore MOE grant (A-8001951-00-00)}}
\date{}

\begin{document}
\maketitle

\begin{abstract}
Tree-child networks are a useful class of binary phylogenetic networks.
The Pons--Batle identity (Pons and Batle, \textit{Scientific Reports}, 2021)
states that the number $a_{n,k}$ of tree-child networks with $k$
reticulations on $n$ taxa satisfies
\[
a_{n,k}=(n-k+1)a_{n,k-1}
+\frac{n(2n+k-3)}{n-k}a_{n-1,k}.
\]
In this paper, we present a short combinatorial proof of this identity.
\end{abstract}

\section{Introduction}

With an increasing number of genomes becoming available, phylogenetic networks have been increasingly used in place of phylogenetic trees to model genome evolution, where reticulate genetic transfer events are common. Consequently, a rich mathematical theory of phylogenetic networks has been developed in the past two decades \cite{huson_book, steel_book}. In particular, tree-child networks \cite{Cardona_09b} and several other classes of phylogenetic networks have been extensively studied \cite{huson_book, Zhang2019clusters}. Tree-child networks are useful in designing algorithms for inferring phylogenetic networks \cite{zhang2026phylofusion} and in visualizing structural variation among phylogenetic trees.

To count tree-child networks with  $k$ reticulations on  $n$ taxa, Pons and Batle conjectured that the count $a_{n, k}$ satisfies the following identity \cite{Conjecture_Pons}:
\begin{eqnarray}
a_{n,k}=(n-k+1)a_{n,k-1}
+\frac{n(2n+k-3)}{n-k}a_{n-1,k}.
\label{conjecture}
\end{eqnarray}
A long and complicated proof of this elegant identity was given by Lin {\it et al.} \cite{Conjecture_proof} through the enumeration of Young tableaux with walls \cite{liu2026combinatorial}. More recently, the Pons--Batle identity has been used in the asymptotic enumeration of binary phylogenetic networks without constraints, normal networks, and galled networks \cite{YFuchs_2026,YZhang_2026}.

In this paper, we present a short combinatorial proof of the Pons--Batle identity by introducing a novel root-decomposition technique for counting tree-child networks.

\section{Preliminaries}\label{sec:1}

 A \textit{binary phylogenetic network}  on a set of taxa  is a rooted directed acyclic graph with no parallel edges consisting of:
\begin{itemize}
    \item a unique {\it root} of indegree 0 and outdegree 1,
    \item $n$ \textit{leaves} of indegree 1 and outdegree 0, labeled one-to-one with the taxa, 
    \item {\it tree nodes} are of indegree 1 and outdegree 2, and
    \item {\it reticulation nodes} of indegree 2 and outdegree 1. 
\end{itemize}
For a binary phylogenetic network $N$, we use $L(N), T(N)$ and $R(N)$ to denote the sets of labeled leaves,  tree nodes and  reticulation nodes, respectively. We also let $V(N)$ denote the set of all vertices in $N$, i.e., $V(N)=L(N)\cup T(N)\cup R(N)$,  and let
$E(N)$ denote the set of directed edges.

Let $u,v\in V(N)$. We call $u$ a \textit{parent} of $v$ and $v$ a
\textit{child} of $u$ if $(u,v)\in E(N)$. Two nodes $u$ and $v$ are
called \textit{siblings} if they have a common parent. We say that $v$
is \textit{reachable} from $u$ if there is a directed path from $u$ to
$v$ consisting of at least one edge.

 Let $(u, v)\in E(N)$. It is called a \textit{reticulation edge} if
 $v\in R(N)$ and  a \textit{tree edge} if $v\in T(N)\cup L(N)$.


\begin{figure}[!t]
    \centering
    \includegraphics[width=0.2\textwidth]{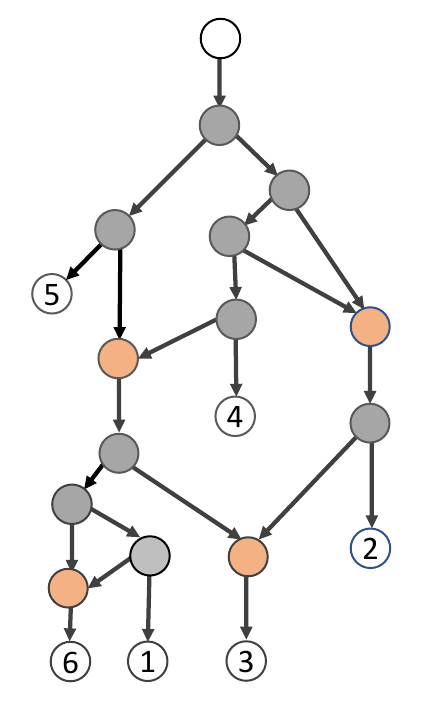}
    \caption{A  tree-child network with four reticulations on taxa $\{1, 2, 3, 4, 5, 6\}$.
    \label{fig1}
    }
\end{figure}

A binary phylogenetic network is called a \emph{tree-child network} if every non-leaf node has a child which is a leaf or a tree node (Figure~\ref{fig1}). We remark that a tree-child network on a single taxon consists only of a root and a leaf.

Let $N$ be a tree-child network. 
The root of $N$ is a {\it free} node. A tree node of $N$ is {\it free} if its two children are  not reticulation nodes. In the tree-child network in Figure~\ref{fig1}, only the top tree node is free.  The edges from a free node to its children are called \textit{free edges}. Then, we have the following simple facts, whose proofs are omitted.

\begin{proposition}\label{prop21}
Let $N$ be a tree-child network with $k$ reticulations on a set of $n$ taxa.
Then $N$ has:
\begin{itemize}
\item $n+k-1$ tree nodes,
\item $n-k-1$ free tree nodes, and
\item $2n+k-1$ tree edges.
\end{itemize}
Furthermore, $0\leq k<n$.
\end{proposition}






\begin{figure}[!b]
    \centering
    \includegraphics[width=0.4\textwidth]{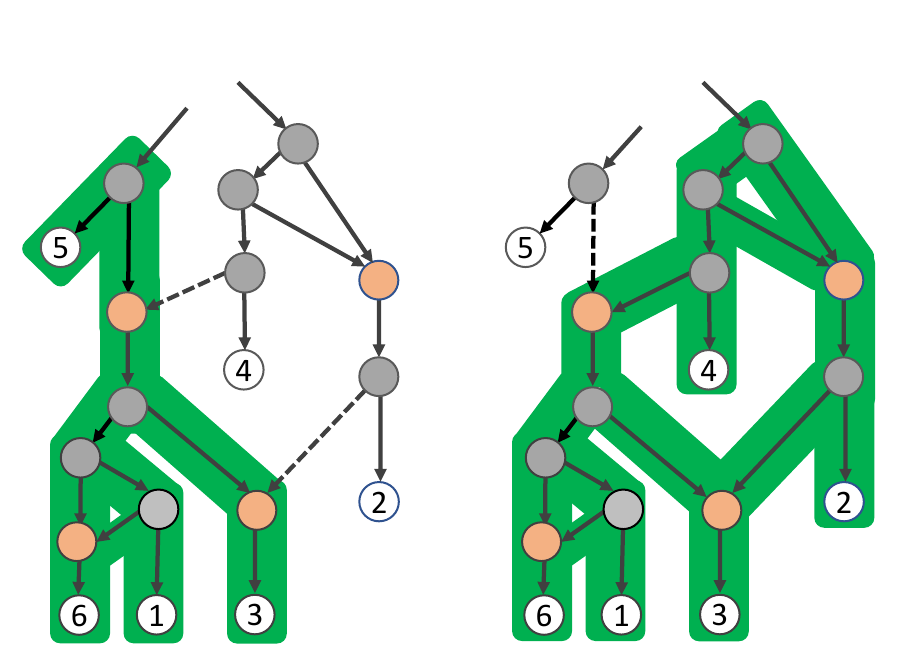}
    \caption{The two possible root splits of the tree-child network in Figure 1. In each root split,  the target subnetwork is colored green and the source subnetwork is the remaining part, where their roots are omitted for clarity.
    \label{fig:NtkExamples}
    }
\end{figure}

\section{Root Decomposition}\label{sec:split}

We use $\TC(n,k)$ to denote the set of tree-child networks with $k$ reticulations on a set of $n$ taxa, and define
$$
a_{n,k}:=|\TC(n,k)|.
$$
By convention, we set $a_{n,k}=0$ unless $n\geq 1$ and $0\leq k<n$. In this section, we establish a recurrence formula for computing $a_{n,k}$ by partitioning a tree-child network into three components.

Throughout this section, we consider $N\in\TC(n,k)$ with $n\geq 2$, and let $w$ denote the unique child of the root.

A node $v$ is said to be reachable from another node $u$ in $N$ if there is a directed path from $u$ to $v$. For a node $u$, define
$$
D(u)=\{u\}\cup\{v: v\text{ is reachable from }u\}.
$$

We define the subnetwork $N_u$ to be the subnetwork induced by $D(u)$; that is, $N_u$ has node set $D(u)$ and directed edge set
$$
\{(x,v)\in E(N): x,v\in D(u)\}.
$$

\begin{definition}[{\bf Root split}]
By the tree-child property, $w$ has at most one reticulation child,
so exactly one of the following holds.
\begin{itemize}
\item[\textbf{T}] Both children of $w$ are non-reticulation nodes.
\item[\textbf{R}] One child of $w$ is a reticulation $\rho$.
Let $y$ be the other child of $w$.
\end{itemize}

In case \textbf{T}, choose one child $x$ of $w$ as the \textit{target node},
and let $y$ be the other child. Define the target part by
\[
P_{tgt}=N_x,
\]
the source part by
\[
P_{src}=\text{the subnetwork induced by }D(y)\setminus D(x),
\]
and the cross edge set by
\[
C_{\rightarrow}
=\{(u,v)\in E(N):u\in D(y)\setminus D(x),\ v\in D(x)\}.
\]
The triple $(P_{src},P_{tgt},C_{\rightarrow})$ is called a \textit{root split}.

In case \textbf{R}, choose the reticulation child $\rho$ of $w$ as the
target node and $y$ as the source node. The root split is defined in
the same way, with $x=\rho$.

Thus, in case \textbf{T}, two distinct root splits of $N$ are obtained,
whereas in case \textbf{R}, only one root split of $N$ is obtained.
\end{definition}

\begin{lemma}[The root split is legal]\label{lem:legal} Each root split of $N$ has the following properties: 
\begin{enumerate} 
\item[\rm(a)] both $P_{tgt}$ and $P_{src}$ are non-empty; 
\item[\rm(b)] there is no directed edge from a node in $P_{tgt}$ to a node in $P_{src}$; 
\item[\rm(c)] if $C_{\rightarrow}$ is non-empty, each of its edges is a reticulation edge in $N$;
\item[\rm(d)] suppressing every node of indegree 1 and outdegree 1 in $P_{src}$ and $P_{tgt}$ results in tree-child networks $A_{src}\in \TC(m,i)$ and $A_{tgt}\in \TC(p,j)$, respectively, where $m\ge1$ and $p\ge1$ such that $m+p=n$; \item[\rm(e)] Let $c=\vert C_{\rightarrow}\vert$. Then, $i+j+c=k$. In addition,  $0\leq c\leq p-j-1$ in case \textbf{T}, whereas $0\leq c\leq p-j$ in case \textbf{R}.
\end{enumerate} 
\end{lemma}
\begin{proof}
Fix a root split. We first consider case \textbf{T}.

(a) Recall that the target node is $x$ and the other child of $w$ is $y$.
Since $x\in D(x)$, $P_{tgt}$ is non-empty. Since $x$ and $y$ are the two
children of $w$, $y$ is not reachable from $x$, and therefore
$y\notin D(x)$. Since $y\in D(y)$, it follows that
$D(y)\setminus D(x)\neq\emptyset$. Hence, $P_{src}$ is non-empty.

(b) Assume that $(u,v)\in E(N)$ and that $u$ is a node in $P_{tgt}$.
By definition, $u$ is reachable from $x$, and hence $v$ is also
reachable from $x$. Thus, $v\in D(x)$, so $v$ is also a node in $P_{tgt}$.
Therefore, there is no directed edge from $P_{tgt}$ to $P_{src}$.

(c) Assume that $(u,v)\in C_{\rightarrow}$. Since $v$ is a node
of $P_{tgt}$, $v$ is reachable from $x$. Hence, there is an edge $e$
entering $v$ in $P_{tgt}$. Since $(u,v)$ and $e$ are two distinct edges
entering $v$ in $N$, $v$ is a reticulation node and hence $(u,v)$ is a reticulation edge of $N$.

(d) Acyclicity is inherited by both $P_{tgt}$ and $P_{src}$.

Any node suppressed in forming $A_{tgt}$ is a reticulation node of $N$
that has lost one of its two incoming edges. Let $r$ be a reticulation
node of $A_{tgt}$. Then $r$ is also a reticulation node of $N$. Its parents,
sibling, and child in $N$ belong to $P_{tgt}$. By the tree-child property,
the sibling and child of $r$ are not reticulation nodes. Hence, the
tree-child property is preserved in $A_{tgt}$. Thus, $A_{tgt}$ is tree-child.

Now let $u$ be a node of indegree $1$ and outdegree $1$ in $P_{src}$.
Then one child of $u$ in $N$ lies in $P_{tgt}$, and the corresponding edge
belongs to $C_{\rightarrow}$. By (c), that child is a reticulation
node. Since $N$ is tree-child, the other child of $u$ is not a
reticulation node. Suppressing $u$ therefore preserves the tree-child
property at the parent of $u$, and does not affect the children of any
other node. Thus, $A_{src}$ is tree-child.

Since $P_{src}$ and $P_{tgt}$ are non-empty, so are $A_{src}$ and $A_{tgt}$.
Therefore, $m\geq1$ and $p\geq1$.

Let $\ell\in L(N)$. Since $x$ and $y$ are the children of $w$,
$\ell$ is reachable from at least one of them. If $\ell$ is reachable
from $x$, then $\ell$ belongs to $A_{tgt}$. Otherwise, $\ell$ belongs to
$A_{src}$. Hence, the leaves of $N$ are partitioned between $A_{src}$ and
$A_{tgt}$, and therefore
\[
m+p=n.
\]

(e) Let $c=|C_{\rightarrow}|$. Each crossing edge enters a
reticulation node that is suppressed in forming $A_{tgt}$, and each
reticulation node suppressed in forming $A_{tgt}$ has exactly one incoming
edge in $C_{\rightarrow}$. Hence, exactly $c$ reticulation nodes of
$N$ disappear when $P_{tgt}$ is reduced to $A_{tgt}$. All other reticulation
nodes of $N$ occur in either $A_{src}$ or $A_{tgt}$. Therefore,
\[
i+j+c=k.
\]

Clearly, $c\geq0$. To prove the upper bound, let $r$ be a reticulation
node suppressed in forming $A_{tgt}$, and let $u$ be its parent in $P_{tgt}$.
The sibling and child of $r$ are not reticulation nodes, by the
tree-child property. After $r$ is suppressed, $u$ is therefore a free
node of $A_{tgt}$. Distinct suppressed reticulation nodes give distinct
free nodes of $A_{tgt}$. Hence, by Proposition~\ref{prop21},
\[
c\leq \#(\text{free nodes in }A_{tgt})=p-j-1.
\]

We now consider case \textbf{R}. The target node is $\rho$ and the
source node is $y$. Since $y$ is not reachable from $\rho$, both
$P_{tgt}$ and $P_{src}$ are non-empty. The same arguments as above show that
properties (a)--(d) also hold in this case. 
For property (e),   one cross edge has $\rho$ as its head in this case and
it is possible that $c=p-j$, as shown in Figure~\ref{fig3}.
\end{proof}

\begin{figure}[!b]
    \centering
    \includegraphics[width=0.4\textwidth]{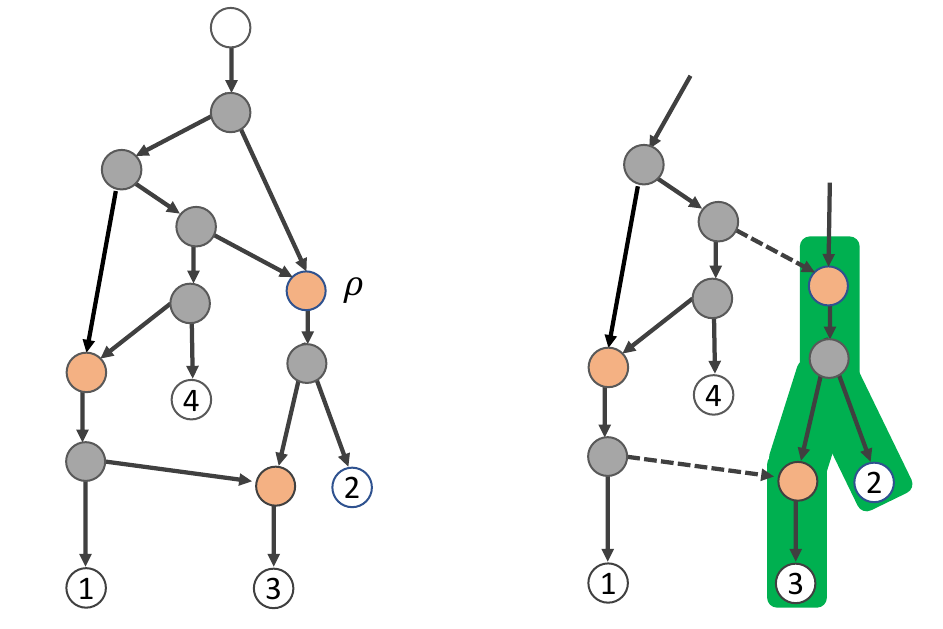}
    \caption{
The unique root split of the tree-child network shown on the left, in which one grandchild $\rho$ of the root is a reticulation node. In this root split, $\rho$ is the head of a cross edge, and both the target subnetwork (green) and the source subnetwork are non-empty, as shown on the right. The roots of the subnetworks are omitted for clarity.
    \label{fig3}
    }
\end{figure}

\begin{definition}[The kernel]\label{def:kernel}
For $m+p=n$,  $i+j+c=k$, $A_{src}\in \TC(m, i)$, and $A_{tgt}\in \TC(p, j)$,   put
\begin{eqnarray*}
  s:= \#(\text{tree edges in }  A_{src}) =2m+i-1, \qquad t:= \#(\text{free nodes in }  A_{tgt} )=p-j,
\end{eqnarray*}
 and, with $\rise{x}{c}:=x(x+1)\cdots(x+c-1)$ the rising
factorial ($\rise{x}{0}=1$),
\begin{eqnarray}
  Q_0:=\left[2^c{p-j-1\choose c}+2^{c-1}{p-j-1\choose c-1}\right],\\
  K_0:=\binom nm\,2^{\,c}\binom{p-j}{c}\,\rise{s}{c}.
\end{eqnarray}
\end{definition}

\begin{lemma}\label{lem:trace}
Keep the notation of Lemma~\ref{lem:legal} and Definition~\ref{def:kernel}. Let
$A_{src}\in\TC(m,i)$ and $A_{tgt}\in\TC(p,j)$, where $m,p\geq1$, and let
$c\leq p-j$ and $k=i+j+c$. Then
{\color{red} $Q_0\rise{s}{c}$ distinct tree-child networks} with $k$ reticulations on $m+p$ taxa are obtained
from $A_{src}$ and $A_{tgt}$ by inserting $c$ cross edges as follows.

Let
\[
\Pi=\{\text{tree edges of }A_{src}\},
\qquad
\Upsilon=\{\text{free nodes of }A_{tgt}\}.
\]
Choose $c$ distinct nodes of $\Upsilon$, and at each chosen node choose
one of its free edges. 

Process the $c$ chosen target edges in a fixed order. For each chosen
target edge $e_2=(u_2,v_2)$, choose an edge
$e_1=(u_1,v_1)\in\Pi$, subdivide $e_1$ by a new tree node $g$,
subdivide $e_2$ by a new reticulation node $h$, and add the cross edge
$(g,h)$. Then replace
\[
\Pi
\quad\text{by}\quad
(\Pi\setminus\{e_1\})\cup\{(u_1,g),(g,v_1)\}.
\]

Finally, a tree-child network is obtained by merging the roots of $A_{src}$ and $A_{tgt}$ into a new tree node $w$
and adding a new root as the parent of $w$.
\end{lemma}

\begin{proof}
If each chosen node of $\Upsilon$ is a free tree node, there are two free edges to choose. If the root is chosen, there is only one outgoing tree edge to choose. Thus, there are   
\[
2^c\binom{p-j-1}{c}+2^{c-1}{p-j-1\choose c-1}=Q_0
\]
choices of target edges.
For each of these 
choices of target edges, we have
\[
|\Pi|=2m+i-1
\]
at the beginning of the process, and each insertion increases $|\Pi|$ by one.
Hence, the numbers of choices for the successive source edges are
\[
2m+i-1,\ 2m+i,\ldots,2m+i+c-2,
\]
whose product is $s^{\overline c}$.
Thus, the edge-insertion process has
$Q_0\rise{s}{c}$
possible outcomes. Since every inserted cross edge is directed from the
source part to the target part, each produced network is acyclic.

We next show that each insertion preserves the tree-child property.
Suppose that $e_1=(u_1,v_1)$ is the chosen tree edge in the source part
and that $e_2=(u_2,v_2)$ is the chosen free edge in the target part.
Subdivide $e_1$ by the new tree node $g$, subdivide $e_2$ by the new
reticulation node $h$, and add the cross edge $(g,h)$.

In the source part, the new tree node $g$ has $v_1$ as one of its
children. Since $e_1$ is a tree edge, $v_1$ is a tree node or a leaf.
Thus, $g$ satisfies the tree-child property. 
$u_1$ still satisfies the tree-child property, as $g$ is a tree node. The insertion does not
change the children of any other node in the source part.

In the target part, $h$ has $v_2$ as its unique child. Since $e_2$ is
a free edge, $v_2$ is not a reticulation node. Hence, $h$ satisfies the
tree-child property. It remains only to check the parent $u_2$ of $h$.

If $u_2$ is a free tree node, let $v'_2$ be its other child. Since
$u_2$ is free, both $v_2$ and $v'_2$ are non-reticulation nodes.
After the insertion, the children of $u_2$ are $h$ and $v'_2$.
Therefore, $u_2$ still has a non-reticulation child and satisfies the
tree-child property.

If $u_2$ is the root of $A_{tgt}$, then after the insertion its unique
child is $h$. At the end of the construction, the roots of $A_{src}$ and
$A_{tgt}$ are merged into the new tree node $w$. The node $w$ has, on the
source side, a non-reticulation child. Hence, $w$ satisfies the
tree-child property in the resulting network.

It follows that every network produced by the construction is a
tree-child network.

It remains to show that the $Q_0\rise{s}{c}$ possible choices give distinct
networks. 
Indeed, the tail of each
inserted cross edge has distinct sets of children for the distinct source edge on which it
was inserted. 
Likewise, its head determines the chosen free edge in
the target part. Consequently, the choices made in the insertion
process can be recovered uniquely from the resulting network, so
different choices produce distinct networks.

Therefore, inserting $c$ cross edges from $A_{src}$ to $A_{tgt}$ produces exactly $Q_0\rise{s}{c}$
distinct tree-child networks with $k=i+j+c$ reticulations.
\end{proof}

\begin{proposition} \label{prop:split}
For $n\ge2$ and $0\le k<n$,
\begin{equation}\label{eq:formB}
  2\,a_{n, k}=\sum_{\substack{m+p=n\\ i+j+c=k}} K_0\,a_{m, i}\,a_{p, j}.
\end{equation}
\end{proposition}

\begin{proof}
By Lemma~3.2, every root split of $N$
determines a source network $A_{src}$, a target network $A_{tgt}$, and $c$
cross edges, and reversing the root split gives exactly the
edge-insertion construction of Lemma~3.4.

If $N$ is in case \textbf{T}, it has two root splits, according to
which of the two children of the top tree node is chosen as the target
node. Thus, $N$ is produced twice by insertion of $c$ edges.

If $N$ is in case \textbf{R}, it has only one root split. In this case, the root of $A_{tgt}$ is one of the $c$ chosen free nodes in the construction. 

Fix $m,p,i,j,c$ such that
$
m+p=n,\; i+j+c=k.
$
There are ${n\choose m}a_{m,i}a_{p,j}$
choices of a source network $A_{src}\in\TC(m,i)$ and a target network
$A_{tgt}\in\TC(p,j)$ on complementary subsets of the $n$ taxa.
For each such pair, by Lemma~3.4, insertion of $c$ cross edges gives 
$2^c{p-j-1\choose c}(2m+i-1)^{\overline c}$
distinct tree-child networks that fall in case {\bf T};
and $2^{c-1}{p-j-1\choose c-1}(2m+i-1)^{\overline c}$
distinct tree-child networks that fall in case {\bf R}.
Hence, summing over all admissible $m,p,i,j,c$, and counting each network produced in case R with multiplicity $2$, we obtain
\[
\sum_{\substack{m+p=n\\ i+j+c=k}}
{n\choose m}
2^c\left[{p-j-1\choose c}+{p-j-1\choose c-1}\right](2m+i-1)^{\overline c}
a_{m,i}a_{p,j}
=
\sum_{\substack{m+p=n\\ i+j+c=k}}
K_0a_{m,i}a_{p,j}.
\]

Thus, networks in case T are counted once for each of their two root splits, while networks in case R are counted twice for their unique root split. Therefore, every network in $\mathrm{TC}(n,k)$ is counted exactly twice, proving identity (\ref{eq:formB}).
\end{proof}

\section{Side and Free Leaves in Tree-Child Networks}

Let $N$ be a tree-child network and $\ell\in L(N)$. We call $\ell$  a \emph{child leaf} if the parent of $\ell$ is a reticulation node, a \emph{side leaf} if the sibling of $\ell$ is a reticulation, and \emph{free leaf} otherwise.
The three kinds are disjoint and exhaustive for leaves in a tree-child network. Define
\begin{eqnarray}
  H(n,k):=\sum_{N\in\TC(n,k)}\#\{\text{child leaves}\} +\sum_{N\in\TC(n,k)}\#\{\text{side leaves}\}, \\
  \Gamma(n, k)=\sum_{N\in\TC(n,k)}\#\{\text{free leaves}\}.
\end{eqnarray}
In other words,  $H(n,k)$ counts the pairs $(N,\ell)$ with $\ell$ a child leaf or a side leaf.  $\Gamma(n,k)$
counts the number of pairs $(N,\ell)$ with $\ell$ a free leaf. Thus, 
\begin{equation}\label{eqn44}
  H(n,k)+\Gamma(n,k)=n\,a_{n,k}\qquad (0\leq k < n).
\end{equation}

\begin{lemma}\label{lem12}
For $n\ge2$ and $0\le k<n$, 
\begin{equation}
\label{eqn55}
   \Gamma(n,k)=n(2n+k-3)a_{n-1,k}.
\end{equation}
\end{lemma}

\begin{proof}
 Let $N\in \TC(n,k)$ and let $\ell$ be a free leaf of $N$ with parent $w$. Since the other child of $w$ is a tree node or a leaf, deleting $\ell$ and suppressing $w$ yields a tree-child network with $k$ reticulations on the remaining $n-1$ leaves.

Conversely, choose the omitted label ($n$ ways), a
network $M\in \TC(n-1,k)$ on the other $n-1$ labels. By Proposition~\ref{prop21},  there are $2n+k-3$ tree edges of $M$ to insert the new leaf (subdivide it and attach the new leaf).
The new leaf is free by definition. The two constructions are mutually inverse.
\end{proof}


\begin{theorem}\label{thm:52}
The Pons--Batle identity is equivalent to the following identity:
\begin{equation}\label{eqn66}
  H(n,k)\;=\;k\,a_{n, k}\;+\;(n-k)(n-k+1)\,a_{n,k-1}
\end{equation}
for every $n\ge1$ and $0\le k<n$.
\end{theorem}
\begin{proof}
Assume that (\ref{eqn66}) holds. Adding identities (\ref{eqn55}) and (\ref{eqn66}) and using (\ref{eqn44}), we obtain, for $n\ge2$ and $1\le k<n$,
\[
n(2n+k-3)a_{n-1,k}
+k\,a_{n,k}
+(n-k)(n-k+1)a_{n,k-1}
=n\,a_{n,k}.
\]
Rearranging this identity gives
\[
a_{n,k}
=(n-k+1)a_{n,k-1}
+\frac{n(2n+k-3)}{n-k}a_{n-1,k},
\]
which is precisely the Pons--Batle identity.

Conversely, assuming the Pons--Batle identity, reversing the above argument yields (\ref{eqn66}).
\end{proof}

In the remainder of the paper, we prove identity (\ref{eqn66}) by deriving a recurrence relation using the root-split operation.

\section{A Recurrence Formula for Counting Non-Free Leaves}\label{sec:kernels}

\begin{definition}\label{def:fourkernels}
Retain $s=2m+i-1$ and $t=p-j$, and set
\[
  K_1:=\binom nm 2^{\,c}\binom tc\rise{(s-1)}{c},\qquad
  K_2:=\binom nm 2^{\,c}\binom{t-1}c\rise sc,
\]
\[
  K_3:=\binom nm 2^{\,c}\binom tc\,c\,\rise s{c-1},\qquad
  K_4:=\binom nm 2^{\,c}\binom{t-1}{c-1}\rise sc,
\]
with $K_3=K_4=0$ for $c=0$.
\end{definition}


\begin{lemma}\label{lem:kernelid}
$K_0=K_1+K_3$, $\ K_0=K_2+K_4$, and $c\,K_0=t\,K_4$.
\end{lemma}

\begin{proof}
The first equation is derived from that $\rise sc-\rise{(s-1)}c=\bigl[(s+c-1)-(s-1)\bigr]\rise s{c-1}=c\,\rise s{c-1}$.

The second is derived from that $\binom tc-\binom{t-1}c=\binom{t-1}{c-1}$ (Pascal's identity).

The third is derived from that  $c\binom tc=t\binom{t-1}{c-1}$.
\end{proof}



\begin{proposition} \label{prop:forcing}
For $n\ge2$ and $0\le k<n$,
\begin{equation}\label{eq:forcing}
  2\,H(n,k)=\sum_{\substack{m+p=n\\ i+j+c=k}}
  \Bigl[K_1\,H(m,i)\,a_{p, j}+K_2\,a_{m, i}\,H(p,j)
  +\bigl(m K_3+p K_4\bigr)a_{m, i}\,a_{p, j}\Bigr].
\end{equation}
\end{proposition}

\begin{proof}

Let $A_{src}\in \TC(m, i)$ and $A_{tgt}\in \TC(p, j)$, where 
$m\geq 1$ and $p\geq 1$ such that $m+p=n$. Consider a tree-child network 
$N\in \TC(n, k)$  obtained by 
inserting $c$ edges from $A_{src}$ to $A_{tgt}$, where
 $c=k-i-j \leq p-j$. 
Consider a leaf  $\ell\in L(N)$. It lies in $A_{src}$ or in
$A_{tgt}$, and in either subnetwork it is  free or non-free. We treat the four cases and show in each
when $\ell$ is non-free in $N$.

 (a) \emph{\bf $\ell$ lies in $A_{src}$ and is non-free there.} Then $\ell$ is non-free in $N$. Indeed edge insertion alters $A_{src}$ only by subdividing its tree edges with new tree
nodes. If $\ell$ was a child leaf of $A_{src}$, its parent was a reticulation $\rho'$. If the 
edge $\rho'\to\ell$ is subdivided by the tail of a cross edge,   $\ell$
becomes a side leaf; otherwise it stays a child leaf. 
If $\ell$ was a side leaf of $A_{src}$,
its sibling was a reticulation $\rho'$; the reticulation edge $\parr(\ell)\to\rho'$ is not subdivided, so the sibling can only change if the edge $\parr(\ell)\to \ell$ is
subdivided, and then $\ell$ becomes a side leaf again. 


(b) \emph{\bf $\ell$ lies in $A_{src}$ and is free there.} Its parent and sibling in $A_{src}$ are
non-reticulation nodes. No reticulation node is created on the source side, so $\ell$ can only become
a side leaf, and this happens exactly when the pendant edge  $\parr(\ell)\to \ell$ is subdivided: the lowest branch vertex inserted
there becomes the parent of $\ell$ and its other child is the head of a cross edge, and hence a reticulation node.
Subdividing the edge into $\parr(\ell)$, or the edge from $\parr(\ell)$ to $\sib(\ell)$,
leaves $\ell$ free, the inserted vertex being a non-reticulation. 


(c) \emph{\bf $\ell$ lies in $A_{tgt}$ and is non-free there.} Then,  $\ell$ is non-free in $N$. The target side
is altered only at chosen free nodes, by subdividing one of their outgoing edges. If
$\ell$ was a child leaf its parent is a reticulation, which is not a free node, so the
edge into $\ell$ is untouched. If $\ell$ was a side leaf its parent has a reticulation child
and so is not free, hence is not chosen, and again the edge into $\ell$ and the sibling of
$\ell$ are untouched. 


(d) \emph{\bf $\ell$ lies in $A_{tgt}$ and is free there.} Then,  its parent $v$ has two
non-reticulation children and is a free node.
If $v$ is not
chosen, the children of $v$ are unchanged and $\ell$ stays free. If $v$ is chosen, then
either the edge $v\to\ell$ is subdivided by a new reticulation node, making $\ell$ a child leaf,
or the other child edge of $v$ is subdivided, making the sibling of $\ell$ a reticulation node
and $\ell$ a side leaf. \emph{Both} binary choices make $\ell$ non-free.

 Let $A_{\rm src}\boxtimes A_{\rm tgt}$ denote the multiset of tree-child networks obtained by inserting $c(=k-i-j)$ cross edges from $A_{\rm src}$ to $A_{\rm tgt}$, where a network in case $\mathbf R$ is counted with multiplicity two.
Let $H_{X}$ denote the number of non-free leaves in a tree-child network $X$. The number of free leaves in $X$ is
$\vert L(X)\vert -H_X$.
Taken together, the above four cases imply that
\begin{eqnarray*}
   && \sum_{X\in A_{src}\boxtimes A_{tgt}}H_X \\
   &=& 2^c{t\choose c}\rise {s}{c} H_{A_{src}}
    +(m-H_{A_{src}})\,2^c\,{t\choose c}\,c\rise s{c-1}
    +2^c{t\choose c}\rise {s}{c} H_{A_{tgt}}
    + (p-H_{A_{tgt}})\,2^c\,{t-1\choose c-1}\rise s{c}.
\end{eqnarray*}
Summing over all possible $A_{src}$ and $A_{tgt}$, 
\begin{align*}
  2H(n,k)=\sum\Bigl[&K_0\,H(m,i)\,a_{p, j}+K_0\,a_{m, i}\,H(p,j)\\
  &+K_3\bigl(m\,a_{m, i}-H(m,i)\bigr)a_{p, j}
   +K_4\,a_{m, i}\bigl(p\,a_{p, j}-H(p,j)\bigr)\Bigr],
\end{align*}
and $K_0-K_3=K_1$, $K_0-K_4=K_2$ turn this into \eqref{eq:forcing}.
\end{proof}


\section{Proof of the Pons--Batle Identity}\label{sec:candidate}

\begin{definition}\label{def:cand}
Define
\[
  D(n,k):=(n-k)(n-k+1)\,a_{n,k-1},\qquad U(n,k):=(n-k)\,a_{n,k},
\]
\[
  \Phi(n,k):=k\,a_{n,k}+(n-k)(n-k+1)\,a_{n,k-1}.
\]
\end{definition}


\begin{lemma} \label{lem:twomark}
For $n\ge2$ and $k\ge1$,
\begin{eqnarray}
 2\,D(n,k)=\sum_{\substack{m+p=n\\ i+j+c=k}}
  \Bigl[K_1\,D(m,i)\,a_{p, j}+K_2\,a_{m, i}\,D(p,j)+K_3\,U(m,i)\,a_{p, j}\Bigr]. \label{eq:twomark}
\end{eqnarray}
\end{lemma}

\begin{proof} 
Applying the root-split identity at $(n,k-1)$ gives
\begin{eqnarray*}
2a_{n,k-1}
=
\sum_{\substack{m+p=n\\ \alpha+\beta+r=k-1}}
\binom{n}{m}2^r
\binom{p-\beta}{r}
(2m+\alpha-1)^{\overline r}
a_{m,\alpha}a_{p,\beta}.
\end{eqnarray*}
Hence,
\begin{eqnarray}
2D(n,k)
&={}&
\sum_{\substack{m+p=n\\ \alpha+\beta+r=k-1}}
\binom{n}{m}2^r
\binom{p-\beta}{r}
(2m+\alpha-1)^{\overline r}  \nonumber\\
& &\qquad{}\times
(n-k)(n-k+1)a_{m,\alpha}a_{p,\beta}.
\label{eqn133}
\end{eqnarray}

For a summand in (\ref{eqn133}), put $
x=m-\alpha,\; y=p-\beta-r$.
Since $m+p=n$ and $\alpha+\beta+r=k-1$, we have 
$x+y=n-k+1$.
Therefore,
\[
(n-k)(n-k+1)
=(x+y-1)(x+y)
=x(x-1)+y(y-1)+2xy.
\]
Substituting this into (\ref{eqn133}), we obtain

\begin{eqnarray}
2D(n,k)
&={}&
\sum_{\substack{m+p=n\\ \alpha+\beta+r=k-1}}
\binom{n}{m}2^r
\binom{p-\beta}{r}
(2m+\alpha-1)^{\overline r}  \nonumber \\
& &\qquad{}\times
\bigl[x(x-1)+y(y-1)+2xy\bigr]
a_{m,\alpha}a_{p,\beta}.
\label{eqn144}
\end{eqnarray}

We reindex the three terms in (\ref{eqn144}) separately.

For the term $x(x-1)$, set
\[
i=\alpha+1,\qquad j=\beta,\qquad c=r.
\]
Then $i+j+c=k$ and
$D(m,i)=x(x-1)a_{m,\alpha}$.
Moreover, $2m+\alpha-1=2m+i-2=s-1$.
Thus this part of (\ref{eqn144}) becomes
\[
\sum_{\substack{m+p=n\\i+j+c=k}}
K_1D(m,i)a_{p,j}.
\]

For the term $y(y-1)$, set
\[
i=\alpha,\qquad j=\beta+1,\qquad c=r.
\]
Then $i+j+c=k$. Since $t=p-j=p-\beta-1$, we have
$y=p-\beta-r=t+1-c$.
Using
\[
(t+1-c)(t-c)\binom{t+1}{c}
=
t(t+1)\binom{t-1}{c}
\]
and $D(p,j)=t(t+1)a_{p,\beta}$,
this part becomes
\[
\sum_{\substack{m+p=n\\i+j+c=k}}
K_2a_{m,i}D(p,j).
\]

Finally, for the term $2xy$, set
\[
i=\alpha,\qquad j=\beta,\qquad c=r+1.
\]
Then $i+j+c=k$, $x=m-i$, and $y=t-c+1$.
Using
\[
(t-c+1)\binom{t}{c-1}
=
c\binom{t}{c}
\]
together with $2\cdot2^{c-1}=2^c$, this part becomes
\[
\sum_{\substack{m+p=n\\i+j+c=k}}
K_3U(m,i)a_{p,j},
\]
where $U(m,i)=(m-i)a_{m,i}$.

 Consequently, \eqref{eq:twomark} holds.
\end{proof}

\begin{proposition}\label{prop:candidate}
For $n\ge2$ and $0\le k<n$,
\begin{equation}\label{eq:candrec}
  2\,\Phi(n,k)=\sum_{\substack{m+p=n\\ i+j+c=k}}
  \Bigl[K_1\,\Phi(m,i)\,a_{p, j}+K_2\,a_{m, i}\,\Phi(p,j)
  +\bigl(mK_3+pK_4\bigr)a_{m, i}\,a_{p, j}\Bigr].
\end{equation}
\end{proposition}

\begin{proof}
For $k=0$, both sides vanish. Indeed, $D(n,0)=0$, so
$\Phi(n,0)=0$, and every summand has $c=0$, hence
$K_3=K_4=0$ and $\Phi(m,0)=\Phi(p,0)=0$.

Assume now that $k\ge 1$. Multiplying the root-split identity
\[
2a_{n,k}
=
\sum_{\substack{m+p=n\\ i+j+c=k}}
K_0a_{m,i}a_{p,j}
\]
by $k=i+j+c$ gives
\begin{equation}
2k a_{n,k}
=
\sum_{\substack{m+p=n\\ i+j+c=k}}
(i+j+c)K_0a_{m,i}a_{p,j}.
\label{eqn211}
\end{equation}

We now add   \eqref{eq:twomark} and (\ref{eqn211}). Since
\[
\Phi(n,k)=k a_{n,k}+D(n,k),
\]
the left-hand side is $2\Phi(n,k)$.

For each summand, the source-side terms are
\[
iK_0a_{m,i}+K_1D(m,i)+K_3U(m,i).
\]
Using $K_0=K_1+K_3$ and
$U(m,i)=(m-i)a_{m,i}$, this equals
\[
\begin{aligned}
&K_1\bigl(i a_{m,i}+D(m,i)\bigr)
 +K_3\bigl(i a_{m,i}+(m-i)a_{m,i}\bigr)\\
&\qquad =
K_1\Phi(m,i)+mK_3a_{m,i}.
\end{aligned}
\]

Similarly, the target-side and crossing terms are
\[
jK_0a_{p,j}+K_2D(p,j)+cK_0a_{p,j}.
\]
Since $K_0=K_2+K_4$, this equals
\[
K_2\Phi(p,j)+(jK_4+cK_0)a_{p,j}.
\]
By $cK_0=tK_4$ and $t=p-j$,
\[
jK_4+cK_0=(j+t)K_4=pK_4.
\]
Hence the target-side and crossing terms become
\[
K_2\Phi(p,j)+pK_4a_{p,j}.
\]

Combining the two parts gives identity (\ref{eq:candrec})
as desired.
\end{proof}

\begin{theorem}
Identity (\ref{eqn66}) in Theorem~\ref{thm:52} holds. Hence, the
Pons--Batle identity holds.
\end{theorem}
\begin{proof}
We show $H(n,k)=\Phi(n,k)$ by strong induction on $n$. For $n=1$ the only admissible pair
is $(1,0)$ and both sides are $0$. Let $n\ge2$ and assume the equality for every smaller
leaf number and every reticulation number. By Lemma~\ref{lem:legal}(d) every nonzero
summand of \eqref{eq:forcing} and \eqref{eq:candrec} has $m,p\ge1$ with $m+p=n$, hence
$m,p<n$, so the induction hypothesis gives $H(m,i)=\Phi(m,i)$ and $H(p,j)=\Phi(p,j)$
throughout. Propositions~\ref{prop:forcing} and~\ref{prop:candidate} then have identical
right-hand sides, so $2H(n,k)=2\Phi(n,k)$.
\end{proof}


\begin{thebibliography}{10}

\bibitem{Cardona_09b}
Gabriel Cardona, Francesc Rossello, and Gabriel Valiente.
\newblock Comparison of tree-child phylogenetic networks.
\newblock {\em IEEE/ACM Trans Comput. Biol. and Bioinform.}, 6(4):552--569, 2009.

\bibitem{YFuchs_2026}
Michael Fuchs and Hao Yu.
\newblock A cube-root phase transition in tree-child networks and the enumeration threshold for galled networks.
\newblock {\em arXiv preprint arXiv:2608.20860}, 2026.

\bibitem{huson_book}
Daniel~H Huson, Regula Rupp, and Celine Scornavacca.
\newblock {\em Phylogenetic networks: concepts, algorithms and applications}.
\newblock Cambridge University Press, 2010.

\bibitem{Conjecture_proof}
Zhicong Lin, Feihu Liu, Jiahang Liu, Jing Liu, and Guoce Xin.
\newblock Proof of a conjecture on {Young} tableaux with walls.
\newblock {\em arXiv preprint arXiv:2601.09551}, 2026.

\bibitem{liu2026combinatorial}
Hexuan Liu, Michael Wallner, and Guan-Ru Yu.
\newblock A combinatorial framework for the {Pons}-{Batle} identity: Young tableaux, lattice paths, and limit laws.
\newblock {\em arXiv preprint arXiv:2605.07587}, 2026.

\bibitem{Conjecture_Pons}
Miquel Pons and Josep Batle.
\newblock Combinatorial characterization of a certain class of words and a conjectured connection with general subclasses of phylogenetic tree-child networks.
\newblock {\em Scientific reports}, 11(1):21875, 2021.

\bibitem{steel_book}
Mike Steel.
\newblock {\em Phylogeny: Discrete and Random Processes in Evolution}.
\newblock SIAM, 2016.

\bibitem{YZhang_2026}
Hao Yu and Louxin Zhang.
\newblock Asymptotic counting of binary phylogenetic networks.
\newblock {\em arXiv preprint arXiv:2605.23126}, 2026.

\bibitem{Zhang2019clusters}
Louxin Zhang.
\newblock Clusters, trees, and phylogenetic network classes.
\newblock In {\em Bioinformatics and Phylogenetics: Seminal Contributions of Bernard Moret}, pages 277--315. Springer, 2019.

\bibitem{zhang2026phylofusion}
Louxin Zhang, Banu Cetinkaya, and Daniel~H Huson.
\newblock Phylofusion—fast and easy fusion of rooted phylogenetic trees into rooted phylogenetic networks.
\newblock {\em Systematic Biology}, 75(1):88--99, 2026.

\end{thebibliography}

\end{document}